\documentclass[12pt, reqno]{amsart}
\usepackage{amssymb,amsthm,amsmath}
\usepackage[numbers,sort&compress]{natbib}
\usepackage{amssymb,amsmath}
\usepackage{amsfonts}
\usepackage{mathrsfs}
\usepackage{latexsym}
\usepackage{amssymb}
\usepackage{amsthm}
\usepackage{indentfirst}
\date{\today}

\let\oldsection\section
\renewcommand\section{\setcounter{equation}{0}\oldsection}

\newtheorem{corollary}{Corollary}[section]
\newtheorem{theorem}{Theorem}[section]
\newtheorem{lemma}{Lemma}[section]
\newtheorem{proposition}{Proposition}[section]
\newtheorem{definition}{Definition}[section]

\newtheorem{remark}{Remark}[section]

\allowdisplaybreaks
\begin{document}

\title[Local existence and uniqueness of Naiver-Stokes]{Local existence and uniqueness of strong solutions to the Navier-Stokes equations
with nonnegative density}


\author{Jinkai~Li}
\address[Jinkai~Li]{Department of Computer Science and Applied Mathematics, Weizmann Institute of Science, Rehovot 76100, Israel.}
\email{jklimath@gmail.com}




\begin{abstract}
In this paper, we consider the initial-boundary value problem to the
nonhomogeneous incompressible Navier-Stokes equations. Local
strong solutions are established, for any
initial data $(\rho_0, u_0)\in (W^{1,\gamma}
\cap L^\infty)\times H_{0,\sigma}^1$,
with $\gamma>1$,
and if $\gamma\geq2$, then the strong solution is unique.
The initial density is allowed to be nonnegative, and in particular,
the initial vacuum is allowed. The assumption
on the initial data is weaker than
the previous widely used one that $(\rho_0, u_0)\in (H^1 \cap
L^\infty )\times(H_{0,\sigma}^1 \cap H^2)$, and no
compatibility condition is required.
\end{abstract}

\maketitle

\allowdisplaybreaks

\section{Introduction}
\label{sec1}

The motion of the incompressible fluid in a domain $\Omega$ is governed
by the following nonhomogeneous incompressible Navier-Stokes equations
\begin{eqnarray}
  &\partial_t\rho+u\cdot\nabla\rho=0,\label{eq1}\\
  &\rho(\partial_tu+(u\cdot\nabla)u)-\Delta u+\nabla p=0,\label{eq2}\\
  &\text{div}\, u=0,\label{eq3}
\end{eqnarray}
in $\Omega\times(0,\infty)$, where the nonnegative function
$\rho$ is the density of the fluid, the vector field $u$
denotes the velocity of the flow, and the scaler function $p$ presents the pressure.
%

Since Leray's pioneer work \cite{Leray} in 1934, in which he established
the the global existence of weak solutions to the homogeneous incompressible
Navier-Stokes equations, i.e.\,system (\ref{eq1})--(\ref{eq3}) with positive
constant density, there has been a considerable number of papers devoted to
the mathematical analysis on the incompressible Navier-Stokes equations. A
generalization of Leray's result to
the corresponding nonhomogeneous system, i.e.\,system
(\ref{eq1})--(\ref{eq3}) with variable density, was first made by Antontsev--Kazhikov in \cite{Ka1}, for the case that
the initial density is away from
vacuum, see also the book Antontsev--Kazhikov--Monakhov \cite{Ka2}.
For the case that
the initial density is allowed to have vacuum, the global existence
of weak solutions to system (\ref{eq1})--(\ref{eq3})
was proved by Simon \cite{Simon1,Simon2}
and Lions \cite{Lions1}. However, the uniqueness and smoothness
of weak solutions to the
nonhomogeneous Navier-Stokes equation, even for
the two dimensional case, is still an open problem; note that it is well known that weak solutions
to the two dimensional homogeneous incompressible Navier-Stokes equations
are unique, and are smooth immediately after the initial time, see, e.g., Ladyzhenskaya \cite{LADVIS} and Temam \cite{TEMBOOK}.

Local existence (but without uniqueness)
of strong solutions to the nonhomogeneous
incompressible Navier-Stokes equations was first established by Antontsev--Kazhikov \cite{Ka1}, under the assumption that the initial
density is bounded and away from zero and the initial velocity has
$H^1$ regularity. Local in time strong solutions, which enjoy the uniqueness,
were later obtained by Ladyzhenskaya--Solonnikov \cite{Lad}, Padula \cite{Padula1,Padula2} and Itoh--Tani \cite{Itoh}. Some more advances
concerning the existence and uniqueness of strong solutions, in
the framework of the so-called
critical spaces, to the nonhomogeneous
incompressible Navier-Stokes equations
have been made recently, see, e.g., \cite{DAN1,DAN2,DAN3,DAN4,AB1,AB2,HPZ,PZ,PZZ}.
It should be mentioned that
in all the works \cite{Ka1,Lad,Padula1,Padula2,Itoh,DAN1,DAN2,DAN3,DAN4,AB1,AB2,HPZ,PZ,PZZ}, the initial density is assumed to have positive lower bound, and thus no vacuum is allowed.

For the general case that the initial density is allowed to have vacuum,
Choe--Kim \cite{CHOEKIM} first proved the local existence and uniqueness
of strong solutions to the initial-boundary value problem
of system (\ref{eq1})--(\ref{eq3}), with initial data $(\rho_0, u_0)$
satisfying
\begin{equation}\label{AS1}
  0\leq\rho_0\in H^1\cap L^\infty,\quad u_0\in H^2\cap H_{0,\sigma}^1,
\end{equation}
and the compatibility condition
\begin{equation}
  \label{AS2}
  \Delta u_0-\nabla p_0=\sqrt{\rho_0}g,
\end{equation}
for some $(p_0, g)\in H^1\times L^2$. Since the work \cite{CHOEKIM},
conditions (\ref{AS1})--(\ref{AS2}) and their necessary modifications
are widely used, as the standard
assumptions, in many papers concerning the studies of
the existence and uniqueness of
strong solutions, with initial vacuum allowed, to the nonhomogeneous
Navier-Stokes equations and some related systems, such as the
magnetohydrodynamics (MHD) and liquid crystals,
see, e.g., \cite{ZJW,HXDWY,WHW,WHYDSJ,CQTZWYJ,LJK,GHJLJK,GHJLJKXC}.

Noticing that, when the initial vacuum is taken into
consideration, conditions (\ref{AS1})--(\ref{AS2}) are so widely used
in the literatures to study the existence and uniqueness
of strong solutions to the nonhomogeneous Navier-Stokes
equations and some related models,
we may ask if one can reduce the regularities on the initial data
stated in (\ref{AS1}) and drop the compatibility condition (\ref{AS2}),
so that the result of existence and uniqueness of strong
solutions to the corresponding systems still holds.
As will be indicated in this paper, we can indeed
reduce the regularities of the initial velocity and drop the
compatibility condition, without destroying the existence
and uniqueness, but the prices that we need to pay are the following:
(i) the corresponding strong solutions do
not have as high regularities as those in \cite{CHOEKIM};
(ii) one can only ask for the continuity, at the initial time, of
the momentum $\rho u$, instead of the velocity $u$ itself.

In this paper, we consider the initial-boundary value
problem to system (\ref{eq1})--(\ref{eq3}), defined on a smooth
bounded domain
$\Omega$ of $\mathbb R^3$, and the initial and boundary conditions are
as follows
\begin{eqnarray}
  &u(x,t)=0,\quad x\in\partial\Omega,\label{bc}\\
  &(\rho, \rho u)|_{t=0}=(\rho_0, \rho_0u_0). \label{ic}
\end{eqnarray}
Note that, instead of imposing the initial condition on the velocity
$u$, we impose the initial condition on the momentum $\rho u$. As
will be explained in (ii) of Remark \ref{remark2},
below, generally one can not expect
the continuity of the velocity $u$, up to the initial time,
when the vacuum
appears and the initial data is not sufficiently smooth. 

Throughout this paper, for positive integer $k$ and positive
number $q\in[1,\infty]$, we use $L^q$ and $W^{k,q}$ to denote the standard Lebesgue and Sobolev spaces, respectively, on the domain $\Omega$.
When $q=2$, we use $H^k$, instead of $W^{k,2}$. Spaces $L^2_\sigma$ and $H^1_{0,\sigma}$ are the closures in $L^2$ and $H^1$, respectively,
of the space $C_{0,\sigma}^\infty:=\{\varphi\in C_0^\infty(\Omega)\,|\,\text{div}\,\varphi=0\}$. For simplicity,
we usually use $\|f\|_q$ to denote $\|f\|_{L^q}$.

Strong solutions to system (\ref{eq1})--(\ref{eq3}), subject to
(\ref{bc})--(\ref{ic}), are defined as follows.

\begin{definition}\label{def}
Given a positive time $T\in(0,\infty)$, and the initial data
$(\rho_0, u_0)$, with $\rho_0\in W^{1,\gamma}\cap L^\infty$, $\gamma\in(1,\infty)$, and
$u_0\in H_{0,\sigma}^1$.
A pair $(\rho, u)$ is called a strong solution to system
(\ref{eq1})--(\ref{eq3}), subject to (\ref{bc})--(\ref{ic}), on
$\Omega\times(0,T)$, if it has the regularities
\begin{eqnarray*}
  &&\rho\in L^\infty(0,T; W^{1,\gamma}\cap L^\infty)\cap
  C([0,T]; L^\gamma), \\
  &&u\in L^\infty(0,T; H_{0,\sigma}^1)\cap L^2(0,T; H^2), \quad \rho u\in C([0,T]; L^2),\\
  &&\sqrt tu\in L^\infty(0,T; H^2)\cap L^2(0,T; W^{2,6}),\quad\sqrt t\partial_tu\in L^2(0,T; H^1),
\end{eqnarray*}
satisfies system (\ref{eq1})--(\ref{eq3}) pointwisely, a.e.~in $\Omega\times(0,T)$, for some associated pressure function $p\in L^2(0,T; H^1)$, and fulfills the initial condition (\ref{ic}).
\end{definition}

\begin{remark}\label{remark1}
Thanks to the regularities of the strong solutions stated in Definition \ref{def}, by equations (\ref{eq1}) and (\ref{eq2}), one can show that the strong solutions have the following additional regularities
$$
\quad\partial_t\rho\in L^4(0,T; L^\gamma),\quad\sqrt\rho\partial_t u\in L^2(0,T; L^2), \quad\sqrt t\sqrt\rho\partial_t u\in L^\infty(0,T; L^2).
$$
%
\end{remark}

The main result of this paper is the following theorem on the local  existence and uniqueness of strong solutions to system (\ref{eq1})--(\ref{eq3}), subject to (\ref{bc})--(\ref{ic}).

\begin{theorem}
  \label{thmmain}
Let $\Omega$ be a bounded domain in $\mathbb R^3$ with smooth boundary. Suppose that the initial data $(\rho_0, u_0)$ satisfies
\begin{eqnarray*}
  0\leq\rho_0\leq\bar\rho, \quad\rho_0\in W^{1,\gamma}, \quad u_0\in H_{0,\sigma}^1,
\end{eqnarray*}
for some $\gamma\in(1,\infty)$ and $\bar\rho\in(0,\infty)$.

Then, there is a positive time $T_0$, depending only on
$\bar\rho, \Omega$ and $\|\nabla u_0\|_2$, such that system
(\ref{eq1})--(\ref{eq3}), subject to (\ref{bc})--(\ref{ic}),
admits a strong solution $(\rho, u)$, on $\Omega\times(0,T_0)$. Moreover, if $\gamma\in[2,\infty)$, then the strong solution just established
is unique.
\end{theorem}

\begin{remark}\label{remark2}

(i) Through we ask for the $W^{1,\gamma}$ regularity on the initial
density, the only factor of the initial density that influences
the existence time $T_0$ in Theorem \ref{thmmain} is the upper bound.
As will be seen
in the proof of Theorem \ref{thmmain}, such higher regularity
assumption on the initial density, i.e.\,$\rho_0\in W^{1,\gamma}$,
is used only to guarantee the continuity of the momentum at the
initial time and the uniqueness of the solution.

(ii) The regularity assumptions on the initial data in
Theorem \ref{thmmain} are weaker than those in \cite{CHOEKIM}, where
the initial data was assumed to have the regularities stated in (\ref{AS1}).
Note that, the compatibility condition (\ref{AS2}) plays an essential role
in \cite{CHOEKIM}, while in Theorem \ref{thmmain}, no compatibility
condition on the initial data is required, for the local existence and
unique of strong solutions.

(iii) Due to the insufficient smoothness and the absence of the
compatibility conditions on the initial data, and the presence of
vacuum, for the strong solutions $(\rho, u)$ established in Theorem \ref{thmmain}, the quantity $\partial_tu$, viewed as a vector valued
function on the time interval $(0,T_0)$, is not generally integrable on $(0,T_0)$. As a result, one can not expect the continuity of $u$,
up to the initial time. It is because of this that we impose the initial condition on $\rho u$, in stead of $u$, in (\ref{ic}), and correspondingly  ask for the continuity in time of $\rho u$ in Definition \ref{def}.
\end{remark}

The key observation leading us to reduce the assumptions on the initial
data, from those imposed in \cite{CHOEKIM} and widely used in
many other papers to the current version, stated
in Theorem \ref{thmmain}, is that the boundedness of the initial
density and the $H^1$ regularity of the initial velocity is sufficient
to guarantee the $L^1(0,T_0; W^{1,\infty})$ estimate
on the velocity of the solutions to system (\ref{eq1})--(\ref{eq3}).
In order to achieve the $L^1(0,T; W^{1,\infty})$ estimate
of the velocity, the main tool is to perform the $t$-weighted $H^2$
estimate to system (\ref{eq1})--(\ref{eq2}) or its approximated
system, see Proposition \ref{prop3.3}, below, obtaining
$$
\sup_{0\leq t\leq T_0}t(\|\nabla^2 u\|_2^2+\|\sqrt{\rho}\partial_t
u\|_2^2)+\int_0^{T_0}t\|\nabla\partial_tu\|_2^2dt\leq C.
$$
Note that, thanks to the weighted factor $t$, the
constant $C$ in the above estimate is independent
of the $H^2$ norm of the initial velocity. With
the above estimate in hand, one can then successfully obtain the
desired $L^1(0,T_0; W^{1,\infty})$ estimate on the velocity, and
further the regularity estimates on the density, see Proposition
\ref{prop4.1}, below, for the details.
In proving the uniqueness
of strong solutions, the idea of the $t$-weighted estimates is
also used, but in a different manner from above,
see the Gronwall type inequality in Lemma \ref{grownwall},
below.

\begin{remark}
(i) The same argument can be adopted to other similar
systems, including the nonhomogeneous incompressible magnetohydrodynamics (MHD) and the liquid crystals, in the presence of initial vacuum.
Specifically, one can weaken the regularity assumptions and drop the
compatibility conditions on the initial data stated in \cite{ZJW,HXDWY,WHW,WHYDSJ,CQTZWYJ,LJK,GHJLJK,GHJLJKXC},
without destroying the existence and uniqueness of strong solutions; however,
the definitions of the strong solutions in those papers
should be modified accordingly.

(ii) The idea of making use of the $t$-weighted estimate has also been
successfully used in the study of several other incompressible
models, to weaken the regularity assumptions on
the initial data, see Li--Titi \cite{LITITI1} for the Boussinesq equations,
Li--Titi \cite{LITITI2} for a tropical atmosphere model, and
Cao--Li--Titi \cite{CAOLITITI} for the primitive equations.
This idea can be also adopted to the compressible Navier-Stokes
equations, but the argument will be different from and more complicated
than the incompressible case. We will present the details of such kind
result for the compressible Navier-Stokes equations in another paper.
\end{remark}

The rest of this paper is arranged as follows: in Section \ref{sec2}, we collect some preliminary lemmas; in Section \ref{sec3}, we
carry out the Galerkin approximation to system (\ref{eq1})--(\ref{eq3}),
and perform some uniform a priori estimates on the solutions to
the approximated system; the proof of Theorem \ref{thmmain} is given in the last section.

\section{Preliminaries}\label{sec2}

In this section, we state several preliminary lemmas which will be used
in the rest of this paper. We start with the following compactness lemma
due to DiPerna--Lions.

\begin{lemma}[cf. \cite{Lions1}]\label{lemlions}
Let $T$ be a positive time, and assume that $\{(\rho_N, u_N)\}_{N=1}^\infty$ satisfies
\begin{eqnarray*}
  &\rho_N\in C([0,T]; L^1),\quad 0\leq\rho_N\leq C,\mbox{ a.e. on }\Omega\times(0,T), \\
  &\text{div}\,u_N=0,\mbox{ a.e. on }\Omega\times(0,T), \quad \|u_N\|_{L^2(0,T; H^1_{0,\sigma})}\leq C, \\
  &\partial_t\rho_N+\text{div}\,(\rho_Nu_N)=0,\mbox{ in }\mathcal D'(\Omega\times(0,T)), \\
  &\rho_N(0)\rightarrow\rho_0,\mbox{ in }L^1,\quad u_N\rightharpoonup u,\mbox{ in }L^2(0,T; H^1),
\end{eqnarray*}
where $C$ is a positive constant independent of $N$.

Then, $\rho_N$ converges in $C([0,T]; L^p)$, for $1\leq p<\infty$, to the unique solution $\rho$, bounded on $\Omega\times(0,T)$, of
\begin{eqnarray*}
  &\partial_t\rho+\text{div}\,(\rho u)=0,\quad\mbox{in }\mathcal D'(\Omega\times(0,T)), \\
  &\rho\in C([0,T]; L^1),\quad \rho(0)=\rho_0,\mbox{ a.e. in }\Omega.
\end{eqnarray*}
\end{lemma}

The next lemma about the existence, uniqueness and a priori estimates to
the transport equations is standard, see, e.g., \cite{Lad}.

\begin{lemma}\label{lemlad} Let $v\in L^1(0,T;Lip)$ a vector field, such that
$\text{div}\,v=0$, and $v\cdot n=0$ on $\partial\Omega$, where $n$ denotes the outward normal vector on $\partial\Omega$. Let $\rho_0\in W^{1,q}$, with $q\in[1,\infty]$.

Then, the following system
\begin{equation*}
\left\{
\begin{array}{l}
\rho_t+\emph{\textmd{div}}(\rho v)=0,\qquad\mbox{ in }\Omega\times(0,T),\\
\rho|_{t=0}=\rho_0,\qquad\mbox{ in }\Omega,
\end{array}
\right.
\end{equation*}
has a unique solution in $L^\infty(0,T;W^{1,\infty})\cap C([0,T];\cap_{1\leq r<\infty}W^{1,r})$, if $q=\infty$, and in $C([0,T];W^{1,q})$, if $1\leq q<\infty$.

Besides, the following estimate holds
$$
\|\rho(t)\|_{W^{1,q}}\leq e^{\int_0^t\|\nabla v(\tau)\|_\infty d\tau}\|\rho_0\|_{W^{1,q}},
$$
for any $t\in[0,T]$.
\end{lemma}

To determine the pressure associated with the strong solutions, we will use the following two lemmas.

\begin{lemma}[cf. \cite{TEMBOOK}]
  \label{lemp1}
Let $\Omega$ be an open set in $\mathbb R^d$, $d\geq2$, and $f=\{f_1, \cdots, f_d\}$, with $f_i$ being distribution, $i=1,2,\cdots,d$. A necessary and sufficient condition for $f=\nabla p$,
for some distribution $p$, is that $\langle f,\phi\rangle=0$, for any $\phi\in C_{0,\sigma}^\infty(\Omega)$.
\end{lemma}

\begin{lemma}[cf. \cite{TEMBOOK}]
  \label{lemp2}
Let $\Omega$ be a bounded Lipschitz open set in $\mathbb R^d$, $d\geq2$.

(i) If a distribution $p$ has all its first-order derivatives $\partial_ip$, $1\leq i\leq d$, in $L^2(\Omega)$, then $p\in L^2(\Omega)$ and
$$
\|p-p_\Omega\|_{L^2(\Omega)}\leq c(\Omega)\|\nabla p\|_{L^2(\Omega)},
$$
where $p_\Omega=\frac{1}{|\Omega|}\int_\Omega pdx$.

(ii) If a distribution $p$ has all its derivatives $\partial_ip$, $1\leq i\leq d$, in $H^{-1}(\Omega)$, then $p\in L^2(\Omega)$ and
$$
\|p-p_\Omega\|_{L^2(\Omega)}\leq c(\Omega)\|\nabla p\|_{H^{-1}(\Omega)}.
$$

In both cases, if $\Omega$ is any open set in $\mathbb R^d$, then $p\in L_{loc}^2(\Omega)$.
\end{lemma}

Finally, we state and prove a Gronwall type inequality
which will be used to guarantee the uniqueness of strong solutions.

\begin{lemma}\label{grownwall}
Given a positive time $T$ and nonnegative functions $f, g, G$ on $[0,T]$, with $f$ and $g$ being absolutely continuous on $[0,T]$. Suppose that
\begin{eqnarray*}
  \left\{
  \begin{array}{l}
    \frac{d}{dt}f(t)\leq A\sqrt{G(t)}, \\
    \frac{d}{dt}g(t)+G(t)\leq\alpha(t)g(t)+\beta(t)f^2(t), \\
    f(0)=0,
  \end{array}
  \right.
\end{eqnarray*}
a.e.\,on $(0,T)$, where $\alpha$ and $\beta$ are two nonnegative functions, with $\alpha\in L^1((0,T))$ and $t\beta(t)\in L^1((0,T))$.

Then, the
following estimates hold
\begin{eqnarray*}
  f(t)\leq A\sqrt{ g(0)}\sqrt te^{\frac12\int_0^t(\alpha(s)+A^2s\beta(s))ds},
\end{eqnarray*}
and
$$
g(t)+\int_0^tG(s)ds\leq g(0)e^{\int_0^t(\alpha(s)+A^2s\beta(s))ds},
$$
for $t\in[0,T]$, which, in particular, imply $f\equiv0$, $g\equiv0$ and $G\equiv0$, provided $g(0)=0$.
\end{lemma}

\begin{proof}
It follows from the assumption and the H\"older inequality that
\begin{equation}\label{ineq1}
f(t)\leq A\int_0^t\sqrt{G(s)}ds\leq A\sqrt t\left(\int_0^tG(s)ds\right)^{\frac12},
\end{equation}
which, along with the assumption, gives
\begin{eqnarray*}
  \frac{dt}{dt}g(t)+G(t)\leq\alpha(t)g(t)+A^2 t\beta(t)\int_0^tG(s)ds.
\end{eqnarray*}
Setting $\eta(t)=g(t)+\int_0^tG(s)ds$, then it follows from the above inequality that
$$
\eta'(t)\leq(\alpha(t)+A^2 t\beta(t))\eta(t),
$$
which, by the Gronwall inequality, implies
$$
\eta(t)=g(t)+\int_0^tG(s)ds\leq g(0)e^{\int_0^t(\alpha(s)+A^2s\beta(s))ds}.
$$
Thanks to the above estimate, and recalling (\ref{ineq1}), we have
\begin{eqnarray*}
  f(t)\leq A\sqrt{ g(0)}\sqrt te^{\frac12\int_0^t(\alpha(s)+A^2s\beta(s))ds}.
\end{eqnarray*}
This completes the proof.
\end{proof}

\section{Galerkin Approximation}\label{sec3}
In this section, we preform the Galerkin approximation to system (\ref{eq1})--(\ref{eq3}). We first present the approximation scheme, then prove the solvability of the
approximated system, and finally carry out the uniform estimates to the
approximated solutions.
\subsection{The scheme} Let $\{w_i\}_{i=1}^\infty$ be a sequence of eigenfunctions to the following eigenvalue problem of the Dirichlet problem to the Stokes equations in $\Omega$:
\begin{equation}
  \label{3.1}
  \left\{
  \begin{array}{l}
    -\Delta w_i+\nabla p_i=\lambda_iw_i, \\
    \text{div}\,w_i=0, \\
    w_i|_{\partial\Omega}=0,
  \end{array}
  \right.
\end{equation}
where $0<\lambda_1\leq\lambda_2\leq\cdots$, with $\lambda_i\rightarrow\infty$,
as $i\rightarrow\infty$, are the eigenvalues. The sequence
$\{w_i\}_{i=1}^\infty$ can be renormalized in such a way that it is an
orthonormal basis in $L^2_\sigma(\Omega)$.
One can further show that it is an orthogonal basis
in $H_{0,\sigma}^1(\Omega)$, and a basis (but not necessary orthogonal) in $H_{0,\sigma}^1(\Omega)\cap
H^2(\Omega)$,
see, e.g., Ladyzhenskaya \cite{LADVIS}. By the regularity
theory of the Stokes equations, $w_i$ is smooth on $\bar\Omega$.

For any positive integer $N$, we set $X_N=span\{w_1, w_2,\cdots, w_N\}$ as the linear space expanded by $w_i$, $i=1,\cdots,N$. Denote by $\|\cdot\|_{X_N}$ the norm on $X_N$, given by
$$
\|f\|_{X_N}=\left(\sum_{i=1}^Na_i^2\right)^{\frac12}, \quad\mbox{for }f= \sum_{i=1}^Na_iw_i.
$$
Recalling that $\{w_i\}_{i=1}^\infty$ is an orthonormal basis in
$L_\sigma^2(\Omega)$, one can verify that $\|f\|_{X_N}$
is exactly the $L^2(\Omega)$ norm of $f$, for any $f\in X_N$. Note that $X_N$ is a finite
dimension space, all other norms on $X_N$ are equivalent to the norm
$\|\cdot\|_{X_N}$ defined above.

We are going to solve the following system:
\begin{equation}
  \label{s3main}
  \left\{
  \begin{array}{l}
  \partial_t\rho_N+u_N\cdot\nabla\rho_N=0, \\
  (\rho_N(\partial_tu_N+(u_N\cdot\nabla)u_N), w)+(\nabla u_N, \nabla w) =0, \quad\forall w\in X_N, \\
  \rho_N|_{t=0}=\rho_{0N}, \quad u_{N}|_{t=0}=u_{0N},
  \end{array}
  \right.
\end{equation}
where $\{\rho_{0N}\}_{N=1}^\infty$ is a sequence of functions from $C^2(\bar\Omega)$, satisfying
\begin{equation}
  0<\underline\rho\leq\rho_{0N}\leq\bar\rho,\quad\rho_{0N}\in C^2(\bar\Omega), \quad \rho_{0N}\rightarrow\rho_0,\mbox{ in }W^{1,r}(\Omega),
\end{equation}
for some $r\in(3,\infty)$, and $u_{0N}$ is given as
\begin{equation}
  \label{3.6}
  u_{0N}=\sum_{i=1}^N(u_0, w_i) w_i.
\end{equation}

For any $v_N\in C([0,T]; X_N)$, define $\Phi(\tau; x,t)$ as the particle
path which goes along with the velocity field $v_N$ and passes through
point $x$ at time $t$:
$$
\frac{d}{d\tau}\Phi(\tau; x,t)=v_N(\Phi(\tau; x,t),\tau),\quad \Phi(t; x, t)=x.
$$
Note that $v_N\in C([0,T]; C^3(\bar\Omega))$, by the standard theory of the ordinary differential equations, the
particle pass $\Phi(\tau; x,t)$ is at least $C^3$ continuous in $(x,t)$
and $C^1$ continuous in $\tau$, on the domain $\{(\tau, x, t)|\tau,
t\in [0,T], x\in\bar\Omega\}$, in other words we have
$\partial_{x,t}^3\Phi, \partial_\tau\Phi\in
C([0,T]\times\bar\Omega\times[0,T])$. Moreover, $\Phi$ depends
continuously on the velocity $v_N$, and actually by considering the difference system of two particle paths $\bar\Phi$ and $\hat\Phi$,
which pass through the same point $x$ at the same time $t$, but go
along two different velocity fields $\bar v_N$ and $\hat v_N$, respectively, by using the mean value theorem of
differentials and the Gronwall inequality, one can explicitly
deduce that
$$
\|\bar\Phi-\hat\Phi\|_{C([0,T]\times\bar\Omega\times[0,T])}\leq T\|\bar v_N-\hat v_N\|_{C(\bar\Omega\times[0,T])}e^{C\|\bar v\|_{L^1(0,T; Lip(\Omega))}},
$$
for a constant depending only on the domain $\Omega$.

Denote $\rho_N=\rho_{0N}(\Phi(0; x,t))$. Using the fact that
$\Phi(0; x,t)=\Phi^{-1}(t; x, 0)$, where the inverse is with respect to the spatial variable $x$, one can easily check that $\rho_N$ is the unique solution to
\begin{equation}
  \label{3.3}
  \left\{
  \begin{array}{l}
  \partial_t\rho_N+v_N\cdot\nabla\rho_N=0,\\
  \rho|_{t=0}=\rho_{0N}.
  \end{array}
  \right.
\end{equation}
Recalling the regularities of $\Phi$, it is straightforward that
$\rho_N\in C^2(\bar\Omega\times[0,T])$.
We define the map $S_N: C([0,T]; X_N)\rightarrow C^2(\bar\Omega\times[0,T])$ as
$$
v_N\mapsto\rho_N=S_N[v_N],\quad\rho_N\mbox{ is the unique solution to }(\ref{3.3}).
$$
Recalling the continuous dependence on $v_N$ of the particle pass
$\Phi$, the above solution mapping $\rho_N=S_N[v_N]$
is continuous, with respect to $v_N\in C([0,T]; X_N)$.

In order to prove the solvability of system (\ref{s3main}), it suffices to find a solution $u_N\in C([0,T]; X_N)$ to the following system
\begin{equation}
  \label{3.5}
  \left\{
  \begin{array}{l}
  (S_N[u_N](\partial_tu_N+(u_N\cdot\nabla) u_N), w)+(\nabla u_N, \nabla w)=0,\quad\forall w\in X_N, \\
  u_N|_{t=0}=u_{0N},
  \end{array}
  \right.
\end{equation}
where $S_N[u_N]$, as defined before, is the unique solution to system
(\ref{3.3}), with $v_N$ replaced by $u_N$. To this end, we consider the following linearized system
\begin{equation}
  \label{3.7}
  \left\{
  \begin{array}{l}
  (S_N[v_N](\partial_t u_N+(v_N\cdot\nabla) u_N), w)+(\nabla u_N, \nabla w)=0, \quad\forall w\in X_N, \\
  u_N|_{t=0}=u_{0N},
  \end{array}
  \right.
\end{equation}
where $v_N\in C([0,T]; X_N)$ is given. We define a solution mapping $Q_N: C([0,T]; X_N)\rightarrow C([0,T]; X_N)$ as
$$
v_N\mapsto u_N=Q_N[v_N], \quad u_N\mbox{ is the unique solution to }(\ref{3.7}).
$$
As it will be shown later, the mapping $Q_N$ is
well-defined. Therefore, to prove the solvability of system (\ref{3.7}),
and consequently system (\ref{s3main}),
it suffices to look for a fixed point of the mapping $Q_N$ in $C([0,T]; X_N)$.

Given $v_N\in C([0,T]; X_N)$, and denote by $\rho_N=S_N[v_N]$, as before, the unique
solution to system (\ref{3.3}). Then $\rho_N\in C^2(\bar\Omega\times [0,T])$ and $\underline\rho\leq\rho\leq\bar\rho$. Suppose
that $u_N$ has the form
$$
u_N(x,t)=\sum_{i=1}^Nf_{Ni}(t) w_i,
$$
for some unknowns $f_{Ni}\in C([0,T])$, $i=1,2,\cdots, N$. Then (\ref{3.7}) is equivalent to
\begin{equation}
  \label{3.8}
  \left\{
  \begin{array}{l}
    \sum_{j=1}^Na_{ij}^N(t)f_{Nj}'(t)+\sum_{j=1}^Nb_{ij}^N(t)f_{Nj}(t) +\lambda_if_{Ni}=0, \\
    f_{Ni}(0)=(u_0, w_i),  \quad i=1,2,\cdots, N,
  \end{array}
  \right.
\end{equation}
where the coefficients $a_{ij}^N$ and $b_{ij}^N$ are given by
$$
a_{ij}^N(t)=(\rho_Nw_j, w_i), \quad b_{ij}^N(t)=(\rho_N(v_N\cdot\nabla) w_j, w_i).
$$
Rewrite the above system of ordinary differential equations in matrix form as
\begin{equation}
  \label{3.9}
  A_N(t)f_N'(t)+(B_N(t)+\Lambda_N)f_N(t)=0, \quad f_N(t)=(f_{N1}(t),\cdots, f_{NN}(t))^T,
\end{equation}
where $A_N(t)=(a_{ij}^N(t))_{N\times N}$, $B_N(t)=(b_{ij}^N(t))_{N\times N}$ and $\Lambda_N=diag\,(\lambda_1,\cdots, \lambda_N)$.

Since $\rho_N\in C^2(\bar\Omega\times[0,T])$ and $v_N\in C([0,T];
X_N)\subseteq C(\bar\Omega\times[0,T])$, it is clear that $A_N, B_N\in
C([0,T])$. Besides, $A_N$ is nonsingular. Otherwise, there are
constants $\alpha_1,\cdots,\alpha_N$, not all zero, such that
$$
A_N(t)\alpha =0, \quad \alpha=(\alpha_1,\cdots,\alpha_N)^T,
$$
that is
$$
\sum_{j=1}^Na_{ij}^N(t)\alpha_j=\sum_{j=1}^N(\rho_Nw_j, w_i)\alpha_j=\left(\rho_N\sum_{j=1}^N\alpha_jw_j, w_i\right)=0,\quad i=1,\cdots, N.
$$
Multiplying by $\alpha_i$ the $i$-th equality of the above system, and summing up the resultants with respect to $i$ yield
$$
\left(\rho_N\sum_{j=1}^N\alpha_jw_j,\sum_{i=1}^N\alpha_iw_i\right)=0,
$$
wherefrom, recalling that $\rho_N\geq\underline\rho>0$, we get
$\sum_{i=1}^N\alpha_iw_i=0,$
which contradicts to the linearly independency of the basis $\{w_i\}_{i=1}^\infty$.

Thanks to the nondegeneracy of $A_N$, (\ref{3.9}) can be reformed as
\begin{equation}
f_N'(t)+A_N^{-1}(t)(B_N(t)+\Lambda_N)f_N(t)=0. \label{3.8'}
\end{equation}
Since $A_N$ and $B_N$ are continuous on $[0,T]$, so is $A_N^{-1}$, the
solvability of the initial value problem to the above system follows
from the classical theory of the ordinary differential equations.
Therefore, for any given $v_N\in C([0,T]; X_N)$, there is a unique
solution $u_N\in C([0,T]; X_N)$ to (\ref{3.7}), in other words, the
solution mapping $Q_N$ is well-defined. Moreover, noticing that the solution mapping $\rho_N=S_N[v_N]$ to system (\ref{3.3}) is continuous in $v_N\in C([0,T]; X_N)$, it is straightforward that the matrices $A_N$ and $B_N$,
viewed as the functionals of $v_N$, are both continuous in $v_N$, so is
$A_N^{-1}$. Therefore, in view of (\ref{3.8'}), $f_N$ is continuous in
$v_N$, and as a result the mapping $u_N=Q_N[v_N]$ is continuous with
respect to $v_N\in C([0,T]; X_N)$.

\subsection{Solvability of (\ref{s3main})} As mentioned
before, in order to prove the solvability of (\ref{s3main}), it
suffices to find a fixed point to the solution mapping $Q_N$, with
$u_N=Q_N[v_N]$ being the unique solution to the linearized system
(\ref{3.7}). Recall that in the previous subsection, we have
shown that $Q_N$ is a continuous mapping from $C([0,T]; X_N)$ to
itself. We will apply the Brower fixed point theorem for compact continuous
mappings to prove the existence of a fixed point to the mapping $Q_N$.
To this end, recalling that the continuity of $Q_N$ has been proven in
the previous subsection, one still need to verify compactness of $Q_N$,
which, noticing that $X_N$ is a finite dimensional space, is guaranteed
by the following proposition:

\begin{proposition}
  \label{prop3.1}
Let $S_N$ and $Q_N$ be the mappings defined as before. Then, for any $v_N\in C([0,T]; X_N)$, the following hold
\begin{eqnarray*}
  &&\sup_{0\leq t\leq T}\|\sqrt{S_N[v_N]}Q_N[v_N]\|_2^2+2\int_0^T\|\nabla Q_N[v_N]\|_2^2 dt\leq\|\sqrt{\rho_{0N}}u_{0N}\|_2^2,\\
  &&\sup_{0\leq t\leq T}\|\nabla Q_N[v_N]\|_2^2+\int_0^T\|\sqrt{S_N[v_N]}\partial_tQ_N[v_N]\|_2^2dt\leq\|\nabla u_{0N}\|_2^2e^{C_NT\|v_N\|_{C([0,T]; X_N)}^2},
\end{eqnarray*}
where $C_N$ is a positive constant depending only on $N, \bar\rho$ and $\Omega$.
\end{proposition}

\begin{proof}
Denote $\rho_N=S_N[v_N]$ and $u_N=Q_N[v_N]$. Taking $w=u_N$ in (\ref{3.7}), then it follows from integration by parts
and using equation (\ref{3.3}) that
$$
\frac12\frac{d}{dt}\|\sqrt{\rho_N}u_N\|_2^2+\|\nabla u_N\|_2^2=0,
$$
from which, integrating in $t$ yields the first conclusion.

Next, we
prove the second conclusion. Choosing $w=\partial_tu_N$ in (\ref{3.7}), and integration by parts, one obtains
\begin{align*}
  \frac12\frac{d}{dt}\|\nabla u_N\|_2^2&+\|\sqrt{\rho_N}\partial_t u_N\|_2^2=-\int_\Omega\rho_N(v_N\cdot\nabla)u_N\cdot\partial_tu_N dx \\
  \leq&\sqrt{\bar\rho}\|\sqrt{\rho_N}\partial_tu_N\|_2\|v_N\|_\infty\|\nabla u_N\|_2\leq C_N\|v_N\|_{X_N}\|\sqrt{\rho_N}\partial_tu_N\|_2\|\nabla u_N\|_2\\
  \leq&\frac12\|\sqrt{\rho_N}\partial_tu_N\|_2^2+C_N\|v_N\|_{X_N}^2\|\nabla u_N\|_2^2,
\end{align*}
wherefrom, by the Gronwall inequality, the second conclusion follows. In the above, we have used the fact that the $L^\infty$ norm and the norm $\|\cdot\|_{X_N}$ are equivalent, as $X_N$ is a finite dimensional Banach
space.
\end{proof}

Thanks to the above proposition, we can prove the global solvability of system (\ref{s3main}), and we have the following:

\begin{corollary}[Solvability of (\ref{s3main})]\label{s3cor}
For any positive time $T$, there is a unique solution $(\rho_N, u_N)$ to
system (\ref{s3main}), satisfying
$$
\rho_N\in C^2(\bar\Omega\times[0,T]),\quad u_N\in C^2([0,T]; X_N), \quad \underline\rho\leq\rho\leq\bar\rho.
$$
\end{corollary}

\begin{proof}
  As mentioned before, it suffices to find a fixed point to the
  mapping $Q_N$ in $C([0,T]; X_N)$. The regularity of $\rho_N$ has been
  mentioned several times in last subsection, while the regularity of
  $u_N$ can be easily seen from the ordinary differential equations
  (\ref{3.8'}), in view of the fact that $A_N^{-1}, B_N\in C^2([0,T])$,
  which are guaranteed by the regularity of $\rho_N$. Thanks to the
  first conclusion in Proposition \ref{prop3.1}, we have the estimate
  $$
  \|Q_N[v_N]\|_{C([0,T]; X_N)}\leq K, \quad\forall v_N\in C([0,T]; X_N),
  $$
  where $K=K(N,\bar\rho, \|u_0\|_2^2)$ is a positive constant.
  By the second conclusion of Proposition \ref{prop3.1}, the following
  estimate holds
  $$
  \|\partial_tQ_N[v_N]\|_{L^2(0,T; X_N)}\leq C(N,T,K, \|u_0\|_{H^1}^2),
  $$
  for any $v_N\in C([0,T]; X_N)$ subject to $\|v_N\|_{C([0,T]; X_N)}\leq K$.
  Recalling that $X_N$ is a finite dimensional
  Banach space, by the Arzel\'a-Ascoli theorem, the above two estimates imply that $Q_N$ is a compact mapping
  from $\mathcal B_K$ to itself,
  where $\mathcal B_K$ is the closed ball in $C([0,T]; X_N)$. Thanks to this fact, and recalling
  that $Q_N$ is a continuous mapping
  from $C([0,T]; X_N)$ to itself, by the
  Brower fixed point theorem, there is a fixed point in
  $\mathcal B_K$ to the mapping $Q_N$. This completes the proof.
\end{proof}

\subsection{Uniform in $N$ estimates} In this subsection, we will establish some a priori estimates, which are uniform in $N$, in a short time, to the solution $(\rho_N, u_N)$ established in Corollary \ref{s3cor}.

Recall the expression of $u_N=\sum_{j=1}^Nf_{Nj}(t)w_j$.
On the one hand, choosing $w=w_i$ in (\ref{s3main}), one obtains by integration by
parts that
$$
(-\Delta u_N, w_i)=(\nabla u_N, \nabla w_i)=(-\rho_N(\partial_tu_N+(u_N \cdot\nabla)u_N, w_i),
$$
for $i=1, 2,\cdots,N$. On the other hand, recalling (\ref{3.1}), it follows from integration by parts that
\begin{eqnarray*}
  (-\Delta u_N, w_i)&=&\sum_{j=1}^Nf_{Nj}(t)(-\Delta w_j, w_i)=\sum_{j=1}^Nf_{Nj}(t)(\lambda_jw_j-\nabla p_j, w_i)\\
  &=&\sum_{j=1}^Nf_{Nj}(t)\lambda_j\delta_{ij}=\lambda_if_{Ni}(t).
\end{eqnarray*}
Thus, we have
$$
f_{Ni}(t)=-\frac{1}{\lambda_i}(\rho_N(\partial_tu_N+(u_N \cdot\nabla)u_N, w_i).
$$
Thanks to this, and using (\ref{3.1}) again, one deduces
\begin{eqnarray*}
  -\Delta u_N&=&-\sum_{j=1}^Nf_{Nj}\Delta w_j=\sum_{j=1}^N(\lambda_jw_j- \nabla p_j)(-\frac{1}{\lambda_j})(\rho_N\dot u_N, w_j)\\
  &=&\nabla\left(\sum_{j=1}^N\frac{1}{\lambda_j}(\rho_N\dot u_N, w_j)p_j\right)-\sum_{j=1}^N(\rho_N\dot u_N, w_j)w_j,
\end{eqnarray*}
with $\dot u_N=\partial_t u_N+(u_N\cdot\nabla)u_N$, or equivalently
\begin{equation}
  \label{3.3-1}
  \Delta u_N+\nabla P_N=\sum_{j=1}^N(\rho_N(\partial_t u_N+(u_N\cdot\nabla)u_N), w_j)w_j,
\end{equation}
where the pressure $P_N$ is given as $P_N=\sum_{j=1}^N\frac{1}{\lambda_j}(\rho_N\dot u_N, w_j)p_j.$

We first consider the $H^1$ estimate, that is the following proposition:

\begin{proposition}
  \label{prop3.2}
  Let $(\rho_N, u_N)$ be the solution established in Corollary \ref{s3cor}. Then, there is a positive time
  $T_0$ depending only on $\bar\rho, \Omega$ and $\|\nabla u_0\|_2$,
  such that
  $$
  \sup_{0\leq t\leq T_0}\|\nabla u_N\|_2^2+\int_0^{T_0}(\|\sqrt{\rho_N} \partial_tu_N\|_2^2+\|\nabla^2 u_N\|_2^2)dt\leq C\|\nabla u_0\|_2^2,
  $$
  for a positive constant $C$ depending only on $\bar\rho$ and $\Omega$.
\end{proposition}

\begin{proof}
Taking $w=\partial_t u_N$ in (\ref{s3main}), then it follows from integration by parts and the Young inequality that
\begin{eqnarray*}
  \frac12\frac{d}{dt}\|\nabla u_N\|_2^2+\|\sqrt{\rho_N}\partial_t u_N\|_2^2=-\int_\Omega\rho_N(u_N\cdot\nabla)u_N\cdot\partial_t u_N dx\\
  \leq\frac12\|\sqrt{\rho_N}\partial_tu_N\|_2^2+\frac12\int_\Omega\rho_N |u_N|^2|\nabla u_N|^2dx,
\end{eqnarray*}
and thus
\begin{equation}
  \label{3.3-2}
  \frac{d}{dt}\|\nabla u_N\|_2^2+\|\sqrt{\rho_N}\partial_tu_N\|_2^2 \leq\int_\Omega\rho_N|u_N|^2|\nabla u_N|^2dx.
\end{equation}

Applying the $H^2$ estimate to (\ref{3.3-1}), and noticing that $\left\|\sum_{j=1}^N(g, w_i)w_i\right\|_2\leq\|g\|_2,$
for any $g\in L^2(\Omega)$, we deduce
\begin{eqnarray}
  \|\nabla^2 u_N\|_2^2&\leq& C\left\|\sum_{j=1}^N(\rho_N (\partial_tu_N+u_N\cdot\nabla u_N), w_j)w_j\right\|_2^2\nonumber\\
  &\leq& C\|\rho_N(\partial_tu_N+u_N\cdot\nabla u_N)\|_2^2\nonumber\\
  &\leq& C\bar\rho\|\sqrt{\rho_N}\partial_tu_N\|_2^2+C\bar\rho\int_\Omega \rho_N|u_N|^2|\nabla u_N|^2dx\nonumber\\
  &\leq& M_1\|\sqrt{\rho_N}\partial_tu_N\|_2^2+C\int_\Omega\rho_N|u_N|^2| \nabla u_N|^2dx, \label{3.3-3}
\end{eqnarray}
where $M_1$ and $C$ are positive constants depending only on $\bar\rho$ and $\Omega.$

Multiplying (\ref{3.3-2}) by $2M_1$, and summing the resultant with (\ref{3.3-3}), we obtain
\begin{align}
  \label{3.3-4}
  2M_1\frac{d}{dt}\|\nabla u_N\|_2^2 +M_1\|\sqrt{\rho_N}\partial_t u_N\|_2^2+\|\nabla^2u_N\|_2^2 \leq C\int_\Omega\rho_N|u_N|^2| \nabla u_N|^2dx,
\end{align}
for a positive constant $C$ depending only on $\bar\rho$ and $\Omega.$

We have to estimate the term $\int_\Omega\rho_N|u_N|^2| \nabla u_N|^2dx$. By the H\"older, Sobolev and Poincar\'e inequalities, we deduce
\begin{eqnarray}
C\int_\Omega\rho_N|u_N|^2|\nabla u_N|^2dx\leq C\|u_N\|_6^2\|\nabla u_N\|_2\|\nabla u_N\|_6\nonumber\\
  \leq C\|\nabla u_N\|_2^3\|\nabla^2 u_N\|_2\leq\frac12\|\nabla^2 u_N\|_2^2+C \|\nabla u_N\|_2^6, \label{3.3-4'}
\end{eqnarray}
which, substituted into (\ref{3.3-4}), gives
\begin{align}
  \label{3.3-5}
  2M_1\frac{d}{dt}\|\nabla u_N\|_2^2+(M_1\|\sqrt{\rho_N}\partial_t
  u_N\|_2^2+\frac12\|\nabla^2 u_N\|_2^2)\leq
  C\|\nabla u_N\|_2^6,
\end{align}
for a positive constant $C$ depending only on $\bar\rho$ and $\Omega$.

Set
$$
F_N(t)=2M_1\|\nabla u_N\|_2^2(t)+\int_0^t(M_1\|\sqrt{\rho_N}\partial_t u_N\|_2^2+\frac12\|\nabla^2 u_N\|_2^2)ds.
$$
Then, it follows from (\ref{3.3-5}) that
$$
F_N'(t)\leq C_1F_N^3(t), \quad t\in[0, T],
$$
where $C_1$ is a positive constant depending only on $\bar\rho$ and $\Omega$. Simple calculations to the above ordinary differential inequality yields
$$
F_N(t)\leq\frac{F_N(0)}{\sqrt{1-2C_1F_N^2(0)t}}=\frac{2M_1\|\nabla u_{0N}\|_2^2}{\sqrt{1-8C_1M_1^2\|\nabla u_{0N}\|_2^4t}},
$$
for any $t\in[0, (16C_1M_1^2\|\nabla u_{0N}\|_2^4)^{-1}]$, from which, noticing that $\|\nabla u_{0N}\|_2\leq\|\nabla u_0\|_2$, one obtains
$$
F_N(t)\leq 2\sqrt2M_1\|\nabla u_{0N}\|_2^2\leq 2\sqrt2M_1\|\nabla u_0\|_2^2, \quad t\in[0,(16C_1M_1^2\|\nabla u_0\|_2^4)^{-1}].
$$
This completes the proof of Proposition \ref{prop3.2}.
\end{proof}

Next, we study the $t$-weighted $H^2$ estimate, which is stated in the next proposition.

\begin{proposition}
  \label{prop3.3}
Let $(\rho_N, u_N)$ be the solution established in Corollary \ref{s3cor} and $T_0$ the number in Proposition \ref{prop3.2}. Then, the following estimate holds
$$
\sup_{0\leq t\leq T_0}t(\|\nabla^2 u_N\|_2^2+\|\sqrt{\rho_N}\partial_t
u_N\|_2^2)+\int_0^{T_0}t\|\nabla\partial_tu_N\|_2^2dt\leq C,
$$
for a positive constant $C$ depending only on $\bar\rho, T_0, \Omega$ and $\|\nabla u_0\|_2$.
\end{proposition}

\begin{proof}
Differentiating $(\ref{s3main})_2$ with respect to $t$, and using $(\ref{s3main})_1$ yield
\begin{align*}
  &(\rho_N(\partial_t^2u_N+(u_N\cdot\nabla)\partial_tu_N), w)+(\nabla\partial_tu_N, \nabla w)\\
  =&
  (\text{div}\,(\rho_Nu_N)(\partial_tu_N+(u_N\cdot\nabla)u_N), w)
  -(\rho_N (\partial_t u_N\cdot\nabla) u_N,w),
\end{align*}
for all $w\in X_N$. Taking $w=\partial_tu_N$ in the above equality, then it follows from integration by parts and using equation $(\ref{s3main})_1$ that
\begin{eqnarray}
  &&\frac12\frac{d}{dt}\|\sqrt{\rho_N}\partial_t u_N\|_2^2+\|\nabla\partial_tu_N\|_2^2\nonumber\\
  &=&(\text{div}\,(\rho_Nu_N)(\partial_tu_N+(u_N\cdot\nabla)u_N)
  -\rho_N\partial_tu_N\cdot\nabla u_N, \partial_tu_N)\nonumber\\
  &\leq&\int_\Omega[\rho_N|u_N|(2|\partial_tu_N||\nabla\partial_tu_N| +|u_N||\nabla u_N||\nabla\partial_tu_N|\nonumber\\
  &&+|u_N||\nabla^2u_N||\partial_tu_N|+|\nabla u_N|^2|\partial_tu_N|)+\rho_N|\partial_tu_N|^2|\nabla u_N|]dx \nonumber\\
  &=&
  2\int_\Omega\rho_N|u_N||\partial_tu_N||\nabla\partial_tu_N| dx+\int_\Omega \rho_N|u_N|^2|\nabla u_N||\nabla\partial_tu_N|dx \nonumber\\
  && +\int_\Omega \rho_N|u_N|^2|\nabla^2u_N||\partial_tu_N|dx+\int_\Omega \rho_N|u_N||\nabla u_N|^2|\partial_tu_N|dx \nonumber\\
  &&+
  \int_\Omega\rho_N|\partial_tu_N|^2|\nabla u_N|dx=:\sum_{i=1}^5I_i.
  \label{3.3-6}
\end{eqnarray}

We estimate $I_i, i=1,2,\cdots,5$, as follows. By the Gagliardo-Nirenberg inequality, $\|f\|_\infty\leq C\|f\|_6^{\frac12}\|f\|_{H^2}^{\frac12}$, it follows from the Sobolev and Poincar\'e inequalities that
\begin{equation}\label{gn}
 \|u_N\|_\infty\leq C\|u_N\|_6^{\frac12}\|u_N\|_{H^2}^{\frac12}\leq C\|\nabla u_N\|_2^{\frac12}\|\nabla^2 u_N\|_2^{\frac12}.
\end{equation}
 Thanks to this, by the H\"older inequality, we can estimate $I_1$ and $I_2$ as
\begin{eqnarray*}
  I_1&\leq&2\|\sqrt{\rho_N} \|_\infty\|u_N\|_\infty\|\sqrt{\rho_N}
  \partial_tu_N\|_2\|\nabla\partial_tu_N \|_2 \\
  &\leq&C\|\nabla u_N\|_2^{\frac12}\|\nabla^2
  u_N\|_2^{\frac12}\|\sqrt{\rho_N}\partial_tu_N\|_2
  \|\nabla\partial_tu_N\|_2,
  \end{eqnarray*}
and
\begin{align*}
  I_2
  \leq\|\rho_N\|_\infty\|u_N\|_\infty^2\|\nabla
  u_N\|_2\|\nabla\partial_t u_N\|_2
  \leq C\|\nabla u_N\|_2^2\|\nabla^2 u_N\|_2\|\nabla\partial_tu_N\|_2,
  \end{align*}
  respectively. By the H\"older, Sobolev and Poincar\'e inequalities, we deuce
  \begin{eqnarray*}
  I_3
  \leq \|\rho_N\|_\infty\|u_N\|_6^2\|\nabla^2u_N\|_2\|\partial_tu_N\|_6
  \leq C\|\nabla u_N\|_2^2\|\nabla^2 u_N\|_2\|\nabla\partial_tu_N\|_2,\\
  I_4
  \leq \|\rho_N\|_\infty\|u_N\|_6\|\nabla u_N\|_2
  \|\nabla u_N\|_6\|\partial_t u_N\|_6
  \leq C\|\nabla u_N\|_2^2\|\nabla^2 u_N\|_2\|\nabla\partial_tu_N\|_2,
\end{eqnarray*}
and
\begin{align*}
    I_5 \leq&\|\sqrt{\rho_N}\|_\infty\|\sqrt{\rho_N}
  \partial_tu_N\|_2\|\nabla u_N\|_3\|\partial_t u_N\|_6 \\
  \leq&C\|\sqrt{\rho_N}\partial_tu_N\|_2\|\nabla
  u_N\|_2^{\frac12}\|\nabla^2 u_N\|_2^{\frac12}
  \|\nabla\partial_tu_N\|_2.
\end{align*}

Substituting the estimates on $I_i, i=1,2,\cdots, 5$, into (\ref{3.3-6}), and using the Young inequality, one obtains
\begin{eqnarray*}
  &&\frac12\frac{d}{dt}\|\sqrt{\rho_N}\partial_tu_N\|_2^2 +\|\nabla\partial_tu_N\|_2^2\nonumber \\
  &\leq&C\|\nabla
  u_N\|_2^{\frac12}\|\nabla^2 u_N\|_2^{\frac12}\|\sqrt{\rho_N}\partial_tu_N\|_2 \|\nabla\partial_tu_N\|_2+C\|\nabla u_N\|_2^2\|\nabla^2 u_N\|_2\|\nabla\partial_tu_N\|_2\\
  &\leq&\frac12 \|\nabla\partial_tu_N\|_2^2+C \|\nabla u_N\|_2\|\nabla^2 u_N\|_2\|\sqrt{\rho_N}\partial_tu_N\|_2^2+C\|\nabla u_N\|_2^4\|\nabla^2 u_N\|_2^2 ,
\end{eqnarray*}
which implies
\begin{eqnarray}
  &&\frac{d}{dt}\|\sqrt{\rho_N}\partial_tu_N\|_2^2 +\|\nabla\partial_tu_N\|_2^2\nonumber \\
  &\leq&C(\|\nabla u_N\|_2\|\nabla^2 u_N\|_2\|\sqrt{\rho_N}\partial_tu_N\|_2^2 +\|\nabla u_N\|_2^4\|\nabla^2 u_N\|_2^2). \label{3.3-7}
\end{eqnarray}

Multiplying the above inequality by $t$ yields
\begin{eqnarray*}
  \frac{d}{dt}(t\|\sqrt{\rho_N}\partial_tu_N\|_2^2) +t\|\nabla\partial_tu_N\|_2^2
  &\leq& C(t\|\nabla u_N\|_2^4\|\nabla^2u_N\|_2^2+\|\sqrt{\rho_N}\partial_tu_N\|_2^2)\\
  &&+C\|\nabla u_N\|_2\|\nabla^2 u_N\|_2t\|\sqrt{\rho_N}\partial_tu_N\|_2^2,
\end{eqnarray*}
from which, by the Gronwall inequality, and using Proposition \ref{prop3.2}, we obtain
\begin{align}
   \sup_{0\leq t\leq T_0}(t\|\sqrt{\rho_N}\partial_tu_N\|_2^2(t))+\int_0^{T_0} t\|\nabla\partial_tu_N\|_2^2(t)dt
  \leq C, \label{3.3-8}
\end{align}
for a positive constant $C$ depending only on $\bar\rho, T_0, \Omega$ and $\|\nabla u_0\|_2$.

Substituting (\ref{3.3-4'}) into (\ref{3.3-3}), one obtains
\begin{equation}
  \|\nabla^2 u_N\|_2^2 \leq C(\|\sqrt{\rho_N}\partial_t u_N\|_2^2+\|\nabla u_N\|_2^6), \label{3.3-9}
\end{equation}
for a positive constant $C$ depending only on $\bar\rho$ and $\Omega$, which along with (\ref{3.3-8}) yields the conclusion.
\end{proof}

\section{Local existence and uniqueness}
\label{seclocpos}
This section is devoted to proving the local existence and uniqueness
of strong solutions to system (\ref{eq1})--(\ref{eq3}), subject to (\ref{bc})--(\ref{ic}), in other words, we will give the proof of
Theorem \ref{thmmain}.

Let us first consider the case that the initial density has positive
lower bound, and we have the following proposition:

\begin{proposition}
  \label{prop4.1}
Suppose that the initial data $(\rho_0, u_0)\in (W^{1,\gamma}\cap L^\infty)\times H_{0,\sigma}^1$, for some $\gamma\in[1,\infty)$, and that $\underline\rho\leq\rho_0\leq\bar\rho$, for two positive constants
$\underline\rho$ and $\bar\rho$, and let $T_0$ be the positive time stated in
Proposition \ref{prop3.3}.

Then, there is a strong solution  $(\rho, u)$ to system (\ref{eq1})--(\ref{eq3}), subject to (\ref{bc})--(\ref{ic}), in $\Omega\times(0,T_0)$, such that $\underline\rho\leq\rho\leq\bar\rho,$ and
\begin{align*}
\sup_{0\leq t\leq T_0}[\|\nabla u\|_2^2&+\|\nabla\rho\|_\gamma^\gamma+ t(\|\nabla^2u\|_2^2+\|\sqrt{\rho}\partial_tu\|_2^2)]+ \int_0^{T_0}[\|\nabla^2u\|_2^2\\
 &+\|\sqrt\rho\partial_tu\|_2^2+ t(\|\nabla^2u\|_6^2+\|\nabla\partial_tu\|_2^2)+\|\nabla u\|_\infty+\|\partial_t\rho\|_\gamma^4] dt\leq C,
\end{align*}
for a positive constant $C$ depending only on $\bar\rho, \Omega, \|\nabla u_0\|_2$ and $\|\nabla\rho_0\|_\gamma$.
\end{proposition}

\begin{proof}
Choose a sequence of $\rho_{0N}\in C^2(\bar\Omega)$, such that
$$
\underline\rho\leq\rho_{0N}\leq\bar\rho, \quad\rho_{0N}\rightarrow\rho_0,\mbox{ in }W^{1,\gamma}(\Omega), \quad\|\nabla\rho_{0N}\|_\gamma\leq\|\nabla\rho_0\|_\gamma.
$$
Let $\{w_i\}_{i=1}^\infty$ be the sequence of eigenfunctions to (\ref{3.1}), as stated in the previous section, Section \ref{sec3}. For any positive integer $N$, we set
$u_{0N}=\sum_{i=1}^N(u_0, w_i)w_i.$
Then, $u_{0N}\rightarrow u_0$ in $H^1(\Omega)$.

By Corollary \ref{s3cor} and Propositions \ref{prop3.2}--\ref{prop3.3}, for any positive integer $N$, there is a solution $(\rho_N, u_N)$ to system (\ref{s3main}), such that  $\underline\rho\leq\rho_N\leq\bar\rho$, and
\begin{align*}
&\sup_{0\leq t\leq T_0}[\|\nabla u_N\|_2^2+t(\|\nabla^2u_N\|_2^2+ \|\sqrt{\rho_N}\partial_tu_N\|_2^2)]\\
&+ \int_0^{T_0}(\|\sqrt{\rho_N}\partial_tu_N\|_2^2+\|\nabla^2u_N\|_2^2 +t\|\nabla\partial_tu_N\|_2^2) dt\leq C,
\end{align*}
where $T_0$ is the positive time stated in Proposition \ref{prop3.2}, and
$C$ is a positive constant depending only on $\bar\rho, T_0, \Omega$ and $\|\nabla u_0\|_2$.

Thanks to the above estimate, using the Cantor diagonal argument, and applying Lemma \ref{lemlions},
there is a subsequence of $\{(\rho_N, u_N)\}_{N=1}^\infty$, still denoted by $\{(\rho_N, u_N)\}_{N=1}^\infty$, and a pair $(\rho, u)$, with $\underline\rho\leq\rho_N\leq\bar\rho$ and
\begin{align}
&\sup_{0\leq t\leq T_0}[\|\nabla u\|_2^2+t(\|\nabla^2u\|_2^2+ \|\sqrt{\rho}\partial_tu\|_2^2)]\nonumber\\
&+ \int_0^{T_0}(\|\sqrt{\rho}\partial_tu\|_2^2+\|\nabla^2u\|_2^2 +t\|\nabla\partial_tu\|_2^2) dt\leq C,\label{new1}
\end{align}
for a positive constant $C$ depending only on $\bar\rho, T_0, \Omega$ and $\|\nabla u_0\|_2$, such that
\begin{eqnarray*}
  &\rho_N\rightarrow\rho,\mbox{ in }C([0,T_0]; L^q),\quad q\in[1,\infty),\\
  &u_N\overset{*}{\rightharpoonup}u,\mbox{ in }L^\infty(0,T_0; H^1)\cap L^\infty(\tau, T_0; H^2),\\
  &u_N\rightharpoonup u,\mbox{ in }L^2(0,T_0; H^2),\quad\partial_tu_N\overset{*}{\rightharpoonup}\partial_tu,\mbox{ in }L^\infty(\tau,T_0; L^2), \\
  &\partial_tu_N\rightharpoonup\partial_tu,\mbox{ in }L^2(0,T_0; L^2)\cap L^2(\tau, T_0; H^1),
\end{eqnarray*}
for any $\tau=\frac{T_0}{k}$, $k=2,3,\cdots$, and thus for any $\tau\in(0,T_0)$, where $\rightharpoonup$ and $\overset{*}{\rightharpoonup}$ denote the weak and weak-* convergences, respectively. Noticing that $H^2\hookrightarrow\hookrightarrow H^1\hookrightarrow L^2$, by the Aubin-Lions compactness lemma, we have $u_N\rightarrow u$ in $C([0,T_0]; L^2)\cap L^2(0,T_0; H^1).$
Therefore, we have
$(\rho, u)|_{t=0}=(\rho_0, u_0).$

Thanks to the previous convergences, it is clear that
$(\rho, u)$ satisfies (\ref{eq1}), in the sense of distribution, and
moreover, since $\rho$ has the regularities $\rho\in L^\infty(0,T_0; W^{1,\gamma})$ and $\partial_t\rho\in L^4(0,T_0; L^\gamma)$, which will be proven in the below, $(\rho, u)$ satisfies equation (\ref{eq1}) pointwisely, a.e.\,in $\Omega\times(0,T_0)$.
The previous convergences also imply
  \begin{eqnarray*}
  &\rho_N\partial_tu_N\rightharpoonup \rho\partial_tu, \mbox{ in }L^2(0,T_0; L^2), \\
  &\rho_N(u_N\cdot\nabla)u_N\rightharpoonup\rho(u\cdot\nabla)u, \mbox{ in } L^2(0,T_0; L^2).
\end{eqnarray*}
Consequently, one can take the limit $N\rightarrow\infty$ in the momentum equation in (\ref{s3main}) to deduce that
$$
(\rho(\partial_tu+(u\cdot\nabla)u), w_i)+(\nabla u, \nabla w_i)=0,\quad\mbox{for any positive integer }i,
$$
or equivalently, by integration by parts, that
$$
(\rho(\partial_tu+(u\cdot\nabla)u)-\Delta u, w_i)=0,\quad\mbox{for any positive integer }i.
$$

Since $\{w_N\}_{N=1}^\infty$ is a basis in $L^2_\sigma(\Omega)$, and
noticing that
$$
\rho(\partial_tu+(u\cdot\nabla)u)-\Delta u\in L^2(0,T_0; L^2),
$$
a density argument yields
$$
(\rho(\partial_tu+(u\cdot\nabla)u)-\Delta u, \phi)=0,\quad\forall w\in L^2_\sigma(\Omega).
$$
Thanks to this, by Lemma \ref{lemp1} and Lemma \ref{lemp2}, there is a  function $\Phi\in L^2(0,T_0; H^1)$, such that
$$
\rho(\partial_tu+(u\cdot\nabla)u)-\Delta u=\nabla\Phi,
$$
which is exactly the momentum equation (\ref{eq2}), by setting $p=-\Phi$.

In order to complete the proof of Proposition \ref{prop4.1}, one still need to verify \begin{align*}
\sup_{0\leq t\leq T_0}\|\nabla\rho\|_\gamma^\gamma+ \int_0^{T_0} (t\|\nabla^2u\|_6^2+\|\nabla u\|_\infty+\|\partial_t\rho\|_\gamma^4)  dt\leq C,
\end{align*}
for a positive constant $C$ depending only on $\bar\rho, T_0, \Omega, \|\nabla u_0\|_2$ and $\|\nabla\rho_0\|_\gamma$.

Similar to (\ref{gn}), one has $\|u\|_\infty^2\leq C \|\nabla u\|_2\|\nabla^2u\|_2.$
Thanks to this, by the elliptic estimate for the Stokes equations, and using the Sobolev inequalitiy, we deduce
\begin{eqnarray*}
  \|\nabla^2u\|_6&\leq&C\|\rho(\partial_tu+u\cdot\nabla u)\|_6\leq C(\|\partial_tu\|_6+\|u\|_\infty\|\nabla u\|_6)\\
  &\leq&C(\|\nabla\partial_tu\|_2+\|\nabla u\|_2^{\frac12}\|\nabla^2u\|_2^{\frac32}),
\end{eqnarray*}
and thus, recalling (\ref{new1}), we obtain
$$
\int_0^{T_0}t\|\nabla^2u\|_6^2 dt\leq C\int_0^{T_0}t(\|\nabla\partial_t u\|_2^2+\|\nabla u\|_2\|\nabla^2u\|_2^3)dt\leq C,
$$
for a positive constant $C$ depending only on $\bar\rho, T_0, \Omega$ and $\|\nabla u_0\|_2$.

By the Gagliardo-Nirenberg inequality, $\|f\|_\infty\leq C(\Omega)\|f\|_6^{\frac12}\|f\|_{W^{1,6}(\Omega)}^{\frac12}$, and using the Sobolev and Poincar\'e inequalities, we have
$$
\|\nabla u\|_\infty\leq C \|\nabla u\|_6^{\frac12}\|\nabla u\|_{W^{1,6}}^{\frac12}\leq C
\|\nabla^2 u\|_2^{\frac12}\|\nabla^2 u\|_6^{\frac12}.
$$
Thus, it follows from the H\"older inequality and (\ref{new1}) that
\begin{eqnarray*}
  \int_0^{T_0}\|\nabla u\|_\infty dt&\leq&C\int_0^{T_0}\|\nabla^2 u\|_2^{\frac12}\|\nabla^2u\|_6^{\frac 12}dt= C\int_0^{T_0}\|\nabla^2 u\|_2^{\frac12}(t\|\nabla^2u\|_6^2)^{\frac 14}t^{-\frac14}dt\\
  &\leq&C\left(\int_0^{T_0}\|\nabla^2u\|_2^2dt\right)^{\frac14} \left(\int_0^{T_0}t\|\nabla^2u\|_6^2 dt\right)^{\frac14}\left(\int_0^{T_0}t^{-\frac12}dt\right)^{\frac12}\leq C,
\end{eqnarray*}
for a positive constant $C$ depending only on $\bar\rho, T_0, \Omega$ and $\|\nabla u_0\|_2$. Thanks to this estimate, and applying Lemma \ref{lemlad}, one obtains
$$
\sup_{0\leq t\leq T_0}\|\nabla\rho\|_\gamma^\gamma\leq C,
$$
for a positive constant $C$ depending only on $\bar\rho, T_0, \Omega, \|\nabla u_0\|_2$ and $\|\nabla\rho_0\|_\gamma$.

Recall that $\|u\|_\infty\leq C(\Omega)\|\nabla u\|_2^{\frac12}\|\nabla^2u\|_2^{\frac12}$, it follows from the continuity equation (\ref{eq1}) that
$$
  \int_0^{T_0}\|\partial_t\rho\|_\gamma^4dt \leq \int_0^{T_0}\|u\|_\infty^4\|\nabla\rho\|_\gamma^4 dt \leq C \int_0^{T_0}\|\nabla u\|_2^2\|\nabla^2u\|_2^2\|\nabla\rho\|_\gamma^4 dt\leq C,
$$
for a positive constant $C$ depending only on $\bar\rho, T_0, \Omega, \|\nabla u_0\|_2$ and $\|\nabla\rho_0\|_\gamma$. This completes the proof of Proposition \ref{prop4.1}.
\end{proof}

We are now ready to prove our main result, Theorem \ref{thmmain}.

\begin{proof}[ { {Proof of Theorem \ref{thmmain}}}]
  Take a sequence $\{\rho_{0n}\}_{n=1}^\infty$, such that
  $$
  \frac 1n\leq\rho_{0n}\leq\bar\rho+1, \quad \rho_{0n}\rightarrow\rho_0\mbox { in }W^{1,\gamma}, \quad \|\nabla\rho_{0n}\|_\gamma\leq\|\nabla\rho_0\|_\gamma+1.
  $$
  By Proposition \ref{prop4.1}, there is a positive time $T_0$ depending only on $\bar\rho, \Omega, \|\nabla u_0\|_2$, such that, for each $n$, there is a strong solution $(\rho_n, u_n)$ to system (\ref{eq1})--(\ref{eq3}), subject to (\ref{bc})--(\ref{ic}), with initial data $(\rho_{0n}, u_0)$, in $\Omega\times(0,T_0)$, satisfying $\frac1n\leq\rho_{0n}\leq\bar\rho+1$ and
\begin{align}
\sup_{0\leq t\leq T_0}&[\|\nabla u_n\|_2^2+\|\nabla\rho_n\|_\gamma^\gamma+ t(\|\nabla^2u_n\|_2^2+\|\sqrt{\rho_n}\partial_tu_n\|_2^2)]+ \int_0^{T_0}[\|\nabla^2u_n\|_2^2\nonumber\\
 &+\|\sqrt{\rho_n}\partial_tu_n\|_2^2+ t(\|\nabla^2u_n\|_6^2+\|\nabla\partial_tu_n\|_2^2)+\|\nabla u_n\|_\infty+\|\partial_t\rho_n\|_\gamma^4] dt\leq C,\label{new2}
\end{align}
for a positive constant $C$ depending only on $\bar\rho, T_0, \Omega, \|\nabla u_0\|_2$ and $\|\nabla\rho_0\|_\gamma$.

  Thanks to the above estimates, by the Cantor diagonal argument, there is a subsequence of $\{(\rho_n, u_n)\}_{n=1}^\infty$, still denoted by $\{(\rho_n, u_n)\}_{n=1}^\infty$, and a pair $(\rho, u)$, such that
  \begin{eqnarray*}
    &u_n\overset{*}{\rightharpoonup}u,\mbox{ in }L^\infty(0,T_0; H^1)\cap L^\infty(\tau,T_0; H^2), \\
    &u_n\rightharpoonup u,\mbox{ in }L^2(0,T_0; H^2)\cap L^2(\tau,T_0; W^{2,6}),\\
    &\partial_tu_n\rightharpoonup\partial_tu,\mbox{ in }L^2(\tau,T_0; H^1), \\
    &\rho_n\overset{*}{\rightharpoonup}\rho,\mbox{ in }L^\infty(0,T_0; W^{1,\gamma}\cap L^\infty), \quad\partial_t\rho_n\rightharpoonup\partial_t\rho,\mbox{ in }L^4(0,T_0; L^\gamma),
  \end{eqnarray*}
  for any $\tau=\frac{T_0}{k}$, $k=1,2,\cdots,$ and thus for any $\tau\in(0,T_0)$. Therefore, by the Aubin-Lions compactness lemma, the following two strong convergences hold
  \begin{eqnarray*}
    &&u_n\rightarrow u, \mbox{ in }C([\tau,T_0]; H^1\cap L^6)\cap L^2(\tau,T_0; C(\bar\Omega)), \\
    &&\rho_n\rightarrow\rho, \mbox{ in }C([0,T_0]; L^q), \quad\mbox{for any }1\leq q<\infty.
  \end{eqnarray*}

  \textbf{Claim 1:} $(\rho, u)$ has the regularities stated in Theorem \ref{thmmain}, and satisfies system (\ref{eq1})--(\ref{eq3}) pointwisely, a.e.\,in $\Omega\times(0,T_0)$.

  The regularities of $(\rho, u)$ stated in Theorem \ref{thmmain}, except $\rho u\in C([0,T_0]; L^2)$, which will be proven in Claim 2, below, follow from the previous weakly convergences. Besides, thanks to the previous convergences, one can check that
  \begin{eqnarray*}
    &&u_n\cdot\nabla\rho_n\rightharpoonup u\cdot\nabla\rho,\mbox{ in }L^2(\tau, T_0; L^\gamma), \\
    &&\rho_n\partial_tu_n\rightharpoonup\rho\partial_tu,\mbox{ in }L^2(\tau,T_0;L^2),\\
    &&\rho_n(u_n\cdot\nabla)u_n\rightharpoonup\rho(u\cdot\nabla)u,\mbox{ in }L^2(\tau,T_0; L^2),
  \end{eqnarray*}
  for any $\tau\in(0,T_0)$. Therefore, one can take the limit $n\rightarrow\infty$ to see that $(\rho, u)$ satisfies equations (\ref{eq1})--(\ref{eq3}), in the sense of distribution, and further a.e.\,in $\Omega\times(0,T_0)$,
  by the regularities of $(\rho, u)$. This proves claim 1.

  \textbf{Claim 2:} $\rho u\in C([0,T_0]; L^2)$ and $(\rho, u)$ satisfies the initial condition (\ref{ic}).

  Since $\rho_n\rightarrow\rho$ in $C([0,T_0]; L^2)$, $\rho_n|_{t=0}=\rho_{0n}$ and $\rho_{0n}\rightarrow\rho_0$ in $W^{1,\gamma}$, one has $\rho|_{t=0}=\rho_0$.
  Recall that $u_n\rightarrow u$ in $C([\tau, T_0]; H^1)$ and $\rho_n\rightarrow\rho$ in $C([0,T_0]; L^q)$, for any $q\in[1,\infty)$, it is clear that
  $\rho_nu_n\rightarrow\rho u, \mbox{ in }C([\tau,T_0]; L^2),$
  for any $\tau\in(0,T_0)$. It remains to verify the continuity of
  $\rho u$ at the original time and the initial data of $\rho u$.

  Similar to (\ref{gn}), one has 
  $\|u_n\|_\infty^2\leq C \|\nabla u_n\|_2\|\nabla^2u_n\|_2.$
  Thanks to this, recalling that $\rho_n\leq\bar\rho+1$, and using (\ref{new2}), it follow from the H\"older that
  \begin{eqnarray*}
    &&\|(\rho_nu_n)(t)-\rho_{0n}u_{0}\|_1
    =\left\|\int_0^t\partial_t(\rho_nu_n)d\tau\right\|_1\\
    &=&\left\|\int_0^t(\rho_n\partial_tu_n+\partial_t\rho_nu_n)d\tau
    \right\|_1
    \leq\int_0^t(\|\rho_n\partial_tu_n\|_1+\|\partial_t\rho_nu_n\|_1)d\tau
    \\
    &\leq& C\int_0^t(\|\sqrt{\rho_n}\|_\infty
    \|\sqrt{\rho_n}\partial_tu_n\|_2+\|\partial_t\rho_n\|_\gamma
    \|u_n\|_\infty)d\tau\\
    &\leq &
    C\int_0^t(\|\sqrt{\rho_n}\partial_tu_n\|_2+\|\partial_t\rho_n
    \|_\gamma \|\nabla
    u_n\|_2^{\frac12}\|\nabla^2u_n\|_2^{\frac12})d\tau\\
    &\leq& C\sqrt t\left[\left(\int_0^t\|\sqrt{\rho_n}\partial_t
    u_n\|_2^2
    d\tau\right)^{\frac12}+\left(\int_0^t\|\partial_t\rho_n\|_\gamma^4
    d\tau\right)^{\frac14}\left(\int_0^t\|\nabla^2u_n\|_2^2
    d\tau\right)^{\frac14}\right]\\
    &\leq& C(\gamma,\bar\rho,T_0,\Omega,\|\nabla u_0\|_2,
    \|\nabla\rho_0\|_\gamma)\sqrt t,
  \end{eqnarray*}
  for any $t\in(0,T_0)$. With the aid of the above estimate, and using
  (\ref{new2}) again, it follows from the H\"older and Sobolev inequalities
  that
  \begin{eqnarray*}
    \|(\rho_nu_n)(t)-\rho_{0n}u_0\|_2&\leq&\|(\rho_nu_n)(t)
    -\rho_{0n}u_n\|_1^{\frac25} \|(\rho_nu_n)(t)
    -\rho_{0n}u_n\|_6^{\frac35}\\
    &\leq&C t^{\frac15}(\|\nabla u_n\|_2(t)+\|\nabla
    u_0\|_2)^{\frac35} \leq C t^{\frac15},
  \end{eqnarray*}
for any $t\in(0,T_0)$, where $C$ is a positive constant independent of $n$.

Thanks to the above estimate, for any $t\in(0,T_0)$, we have
\begin{eqnarray*}
  &&\|(\rho u)(t)-\rho_0u_0\|_2\\
  &\leq&\|(\rho u)(t)-(\rho_nu_n)(t)\|_2+\|(\rho_nu_n)(t)-\rho_{0n}u_0\|_2+\|\rho_{0n} u_0-\rho_0u_0\|_2\\
  &\leq&\|(\rho u)(t)-(\rho_nu_n)(t)\|_2+\|\rho_{0n} u_0-\rho_0u_0\|_2+Ct^{\frac15},
\end{eqnarray*}
for any positive integer $n$, where $C$ is a positive constant independent of $n$.
Recall that $\rho_nu_n\rightarrow\rho u$ in $C([\tau,T_0]; L^2),$ for any $\tau\in(0,T_0)$, and $\rho_n\rightarrow\rho$ in $C([0,T]; L^q)$, for any $q\in[1,\infty)$.
Thus, we have
$$
  \|(\rho u)(t)-\rho_0u_0\|_2\leq\varliminf_{n\rightarrow\infty} (\|(\rho u)(t)-(\rho_nu_n)(t)\|_2+\|\rho_{0n} u_0-\rho_0u_0\|_2)+Ct^{\frac15}
  =Ct^{\frac15},
$$
for any $t\in(0,T_0)$, which implies that $\rho u$ is continuous at the original time and satisfies the initial condition $\rho u|_{t=0}=\rho_0u_0$. This proves Claim 2, and thus further 
the existence part of Theorem \ref{thmmain}.

We now prove the uniqueness part of Theorem \ref{thmmain}. Let $(\tilde\rho, \tilde u)$ and $(\hat\rho, \hat u)$ be two local
strong solutions
to system (\ref{eq1})--(\ref{eq3}), subject to (\ref{bc})--(\ref{ic}), on $\Omega\times(0,T)$, for a positive time $T$, with the same initial data $(u_0, \rho_0)$. Then, we have following regularities
\begin{eqnarray}
  \tilde\rho,\hat\rho\in L^\infty(0,T; H^1),\quad \tilde u,\hat u\in L^2(0,T; H^2),\label{reg1}\\
  \sqrt t\tilde u,\sqrt t\hat u\in L^\infty(0,T; H^2),\quad\sqrt t\partial_t\tilde u,\sqrt t\hat u\in L^2(0,T; H^1).  \label{reg2}
\end{eqnarray}Setting
$$
\rho=\tilde\rho-\hat\rho, \quad u=\tilde u-\hat u,
$$
then $(\rho, u)$ satisfies
\begin{eqnarray}
  &&\partial_t\rho+\tilde u\cdot\nabla\rho+u\cdot\nabla\hat\rho=0,\label{deq1}\\
  &&\text{div}\,u=0,\label{deq2}\\
  &&\tilde\rho(\partial_tu+\tilde u\cdot\nabla u)-\Delta u+\nabla p=-\tilde\rho u\cdot\nabla\hat u -\rho(\partial_t\hat u+\hat u\cdot\nabla\hat u),\label{deq3}
\end{eqnarray}
a.e. in $\Omega\times(0,T)$.

Multiplying equation (\ref{deq1}) by $|\rho|^{-\frac12}\rho$, and integration by parts, then it follows from the H\"older, Sobolev and Poincar\'e inequalities that
\begin{eqnarray*}
  \frac23\frac{d}{dt}\|\rho\|_{3/2}^{3/2}&=&-\int_\Omega u\cdot\nabla\hat\rho|\rho|^{-\frac12}\rho dx\leq \|u\|_6\|\nabla\hat\rho\|_2\|\rho\|_{3/2}^{1/2}\\
  &\leq&C\|\nabla u\|_2 \|\nabla\hat\rho\|_2\|\rho\|_{3/2}^{1/2},
\end{eqnarray*}
from which, recalling (\ref{reg1}), one obtains \begin{equation}
  \frac{d}{dt}\|\rho\|_{3/2} \leq A\|\nabla u\|_2, \label{ndf1}
\end{equation}
for a positive constant $A$.

Same to (\ref{gn}), we have $\|\hat u\|_\infty^2\leq C\|\nabla\hat u\|_2\|\nabla^2\hat u\|_2$. Multiplying equation (\ref{deq3}) by $u$, and using equation (\ref{deq1}), it follows from integration by parts, the H\"older, Sobolev, Poincar\'e and Young inequalities that
\begin{eqnarray*}
  &&\frac12\frac{d}{dt}\|\sqrt{\tilde\rho}u\|_2^2+\|\nabla u\|_2^2
  =-\int_\Omega [\tilde\rho u\cdot\nabla\hat u -\rho(\partial_t\hat u+\hat u\cdot\nabla\hat u)]\cdot u dx\\
  &\leq&\|\sqrt{\tilde\rho}\|_\infty\|\sqrt{\tilde\rho}u\|_2\|u\|_6\|\nabla\hat u\|_3+\|\rho\|_{3/2}(\|\partial_t\hat u\|_6\|u\|_6+\|\hat u\|_\infty \|\nabla\hat u\|_6\|u\|_6)\\
  &\leq& C\|\sqrt{\tilde\rho}u\|_2\|\nabla u\|_2\|\nabla\hat u\|_{H^1}+C\|\rho\|_{3/2}(\|\nabla\partial_t\hat u\|_2\|\nabla u\|_2+\|\nabla\hat u\|_2^{1/2}\|\nabla\hat u\|_{H^1}^{3/2}\|\nabla u\|_2)\\
  &\leq&\frac12\|\nabla u\|_2^2+C\|\nabla\hat u\|_{H^1}^2\|\sqrt{\tilde\rho}u\|_2^2+C(\|\nabla\partial_t\hat u\|_2^2+\|\nabla\hat u\|_2\|\nabla\hat u\|_{H^1}^3)\|\rho\|_{3/2}^2,
\end{eqnarray*}
from which, one obtains
\begin{equation}
  \label{ndf2}
  \frac{d}{dt}\|\sqrt{\tilde\rho}u\|_2^2+\|\nabla u\|_2^2\leq\alpha(t)\|\sqrt{\tilde\rho}u\|_2^2+\beta(t)\|\rho\|_{3/2}^2,
\end{equation}
where
$$
\alpha(t)=C\|\nabla\hat u\|_{H^1}^2(t),\quad\beta(t)=C(\|\nabla\partial_t\hat u\|_2^2+\|\nabla\hat u\|_2\|\nabla\hat u\|_{H^1}^3)(t).
$$

Recalling (\ref{reg2}), it is clear that $\alpha\in L^1((0,T))$ and $t\beta(t)\in L^1((0,T))$. As a result, combining (\ref{ndf1}) and (\ref{ndf2}), and applying Lemma \ref{grownwall}, one obtains $\|\rho\|_{3/2}\equiv\|\sqrt{\tilde\rho} u\|_2\equiv\|\nabla u\|_2\equiv0$. Thus $\rho\equiv u\equiv0$, proving the uniqueness part of Theorem \ref{thmmain}.
\end{proof}



\begin{thebibliography}{50}
\bibitem{AB1}
Abidi, H.; Gui, G.; Zhang, P.: \emph{On the decay and stability of global solutions to the 3D inhomogeneous Navier-Stokes
equations}, Comm. Pure Appl. Math., \bf65 \rm(2011), 832--881.

\bibitem{AB2}
Abidi, H.; Gui, G., Zhang, P.: \emph{On the wellposedness of 3-D inhomogeneous Navier¨CStokes equations in the critical
spaces}, Arch. Ration. Mech. Anal., \bf204 \rm(2012), 189--230.


\bibitem{Ka1}
Antontsev, S.~A.; Kazhikov, A.~V.: \emph{Mathematical study of flows of nonhomogeneous fluids}, in: Lecture notes, Novosibirsk State University, USSR, 1973.

\bibitem{Ka2}
Antontsev, S.~A.; Kazhikov, A.~V.; Monakhov, V.~N.:
\emph{Boundary value problems in mechanics of nonhomogeneous fluids},
North-Holland, Amsterdam, 1990.

\bibitem{CAOLITITI}
Cao, C.; Li, J.; Titi, E.~S.: \emph{Local and global well-posedness of strong solutions to the 3D primitive equations with vertical eddy diffusivity}, Arch. Ration. Mech. Anal., \bf214 \rm(2014), 35--76.

\bibitem{CQTZWYJ}
Chen, Q.; Tan, Z.; Wang, Y.: \emph{Strong solutions to the incompressible magnetohydrodynamic equations}, Math. Methods Appl. Sci., \bf34 \rm(2011), 94--107.

\bibitem{CHOEKIM}
Choe, H.~J.; Kim, H.: \emph{Strong solutions of the Navier-Stokes equations for nonhomogeneous incompressible fluids}, Comm. Partial Differential Equations, \bf28 \rm(2003), 1183--1201.

\bibitem{WHW}
Craig, W.; Huang, X.; Wang, Y.: \emph{Global wellposedness for the 3D inhomogeneous incompressible Navier-Stokes equations}, J. Math. Fluid Mech., \bf15 \rm(2013), 747--758.

\bibitem{DAN3}
Danchin, R.: \emph{Density-dependent incompressible viscous fluids in critical spaces}, Proc. Roy. Soc. Edinburgh Sect. A, \bf133 \rm(2003), 1311--1334.

\bibitem{DAN1}
Danchin, R.; Mucha, P.~B.: \emph{A Lagrangian approach for the incompressible Navier-Stokes equations with variable
density}, Comm. Pure Appl. Math., \bf65 \rm(2012), 1458--1480.

\bibitem{DAN2}
Danchin, R.; Mucha, P.~B.: \emph{Incompressible flows with piecewise constant density}, Arch. Ration. Mech. Anal., \bf207 \rm(2013), 991--1023.

\bibitem{DAN4}
Danchin, R.; Zhang, P.: \emph{Inhomogeneous Navier-Stokes equations in the half-space, with only bounded density}, J. Funct. Anal., \bf267 \rm(2014), 2371--2436.

\bibitem{GHJLJK}
Gong, H.; Li, J.: \emph{Global existence of strong solutions to incompressible MHD}, Commun. Pure Appl. Anal., \bf13 \rm(2014), 1553--1561.

\bibitem{GHJLJKXC}
Gong, H.; Li, J.; Xu C.: \emph{Local well-posedness of strong solutions to density-dependent liquid crystal system}, arXiv:1512.00716

\bibitem{HPZ}
Huang, J.; Paicu, M.; Zhang, P.: \emph{Global solutions to incompressible inhomogeneous fluid system with bounded density
and non-Lipschitz velocity}, Arch. Ration. Mech. Anal., \bf209 \rm(2013), 631--682.

\bibitem{HXDWY}
Huang, X.; Wang, Y.: \emph{Global strong solution of 3D inhomogeneous Navier-Stokes equations with density-dependent viscosity}, J. Differential Equations, \bf259 \rm(2015), 1606--1627.

\bibitem{Itoh}
Itoh, S.; Tani, A.: \emph{Solvability of nonstationary problems for nonhonogeneous incompressible fulids and the convergence with vanishhing viscosity}, Tokyo J. Math., \bf22 \rm(1999), 17--42.

\bibitem{LADVIS}
Ladyzhenskaya, O.~A.: \emph{The Mathematical Theory of Viscous Incompressible Flow}, Second English edition, Mathematics and its Applications, Vol. 2 Gordon and Breach, Science Publishers, New York-London-Paris, 1969.

\bibitem{Lad}
Ladyzhenskaya, O.~A.; Solonnikov, V.~A.: \emph{Unique solvability of an initial and boundary value problem for viscous incompressible non-homogeneous fluids}, J. Soviet. Math., \bf9 \rm(1978), 697--749.

\bibitem{Leray}
Leray, J.: \emph{Sur le mouvement d'um liquide visqueux emplissant l'espace},
Acta Math., \bf63 \rm(1943), 193--248.

\bibitem{LJK}
Li, J.: \emph{Global strong solutions to the inhomogeneous incompressible nematic liquid crystal flow}, Methods Appl. Anal., \bf22 \rm(2015), 201--220.

\bibitem{LITITI1}
Li, J.; Titi, E.~S.: \emph{Global well-Posedness of the 2D Boussinesq equations with vertical dissipation}, Arch. Ration. Mech. Anal., \bf220 \rm(2016), 983--1001.

\bibitem{LITITI2}
Li, J.; Titi, E.~S.: \emph{Global well-posedness of strong solutions to a tropical climate model}, Discrete Contin. Dyn. Syst., \bf36 \rm(2016), 4495--4516.

\bibitem{Lions1}
Lions, P.~L.: \emph{Mathematical Topics in Fluid Mechanics.} Vol. 1. Incompressible models, Oxford Lecture Series in Mathematics and its Applications, 3. Oxford Science Publications. The Clarendon Press,
Oxford University Press, New York, 1996.

\bibitem{Padula1}
Padula, M.: \emph{An existence theorem for non-homogeneous incompressible fluids}, Rend. Cir. Math. Palermo, \bf31(4) \rm(1982), 119--124.

\bibitem{Padula2}
Padula, M.: \emph{On the existence and uniqueness of nonhomogeneous motions in exiterior domains}, Math. Z., \bf203 \rm(1990), 581--604.

\bibitem{PZ}
Paicu, M.; Zhang, P.: \emph{Global solutions to the 3-D incompressible inhomogeneous Navier¨CStokes system}, J. Funct.
Anal., \bf262 \rm(2012), 3556--3584.

\bibitem{PZZ}
Paicu, M.; Zhang, P.; Zhang, Z.: \emph{Global unique solvability of homogeneous Navier-Stokes equations with bounded
density}, Comm. Partial Differential Equations, \bf38 \rm(2013), 1208--1234.

\bibitem{Simon1}
Simon, J.: \emph{Ecoulement d'um fluid non homog\'ene avec une densit\'e initiale s'annulant},
C. R. Acad. Sci. Paris Ser. A, \bf15 \rm(1978), 1009--1012.

\bibitem{Simon2}
Simon, J.: \emph{Nonhomogeneous viscous incompressible fluids: existence of velocity, density, and pressure},
SIAM J. Math. Anal., \bf21(5) \rm(1990), 1093--1117.

\bibitem{TEMBOOK}
Temam, R.: \emph{Navier-Stokes Equations. Theory and Numerical Analysis}, Revised edition, Studies in Mathematics and its Applications, 2., North-Holland Publishing Co., Amsterdam-New York, 1979.

\bibitem{WHYDSJ}
Wen, H.; Ding, S.: \emph{Solutions of incompressible hydrodynamic flow of liquid crystals}, Nonlinear Anal. Real World Appl., \bf12 \rm(2011),  1510--1531.

\bibitem{ZJW}
Zhang, J.: \emph{Global well-posedness for the incompressible Navier-Stokes equations with density-dependent viscosity coefficient, J. Differential Equations}, \bf259 \rm(2015), 1722--1742.
\end{thebibliography}
\end{document}